\documentclass[11pt]{amsart}

\usepackage{amsmath}
\usepackage{amssymb}
\usepackage{graphicx}

\DeclareMathOperator{\vol}{vol}

\newcommand{\Eu}{\mathbb{E}}

\renewcommand{\Re}{\mathbb R}
\newcommand{\BB}{\mathbf{B}}

\newcommand{\M}{\mathbb{M}}
\newcommand{\N}{\mathbb{N}}
\newcommand{\PPP}{\mathbb{P}}
\newcommand{\MM}{\mathbf{M}}
\newcommand{\F}{\mathcal{F}}
\renewcommand{\S}{\mathbb{S}}
\renewcommand{\P}{\mathcal{P}}
\newcommand{\PP}{\mathbf{P}}
\newcommand{\z}{\mathbf{z}}
\newcommand{\y}{\mathbf{y}}
\newcommand{\x}{\mathbf{x}}
\newcommand{\q}{\mathbf{q}}
\newcommand{\p}{\mathbf{p}}
\newcommand{\VV}{\mathbf{V}}
\newcommand{\CC}{\mathbf{C}}
\newcommand{\TT}{\mathbf{T}}
\newcommand{\DD}{\mathbf{D}}

\renewcommand{\o}{\mathbf{o}}

\DeclareMathOperator{\conv}{conv}

\DeclareMathOperator{\area}{area}
\DeclareMathOperator{\bd}{bd}

\DeclareMathOperator{\inter}{int}

\renewcommand{\c}{\mathbf{c}}

\newtheorem{theorem}{Theorem}[section]
\newtheorem{corollary}[theorem]{Corollary}
\newtheorem{lemma}[theorem]{Lemma}

\theoremstyle{definition}
\newtheorem{definition}[theorem]{Definition}

\newtheorem{remark}[theorem]{Remark}
\newtheorem{conjecture}[theorem]{Conjecture}

\numberwithin{equation}{section}

\title[Truncated density bounds for unit ball packings]{Density bounds for unit ball packings relative to their outer parallel domains} 

\author[K. Bezdek]{K\'{a}roly Bezdek}
\address[K. Bezdek]{Department of Mathematics and Statistics, University of Calgary, Canada and Department of Mathematics, University of Pannonia, Vesz\-pr\'em, Hungary}
\email{{\tt kbezdek@ucalgary.ca}}

\author[Z. L\'angi]{Zsolt L\'angi}
\address[Z. L\'angi]{MTA-BME Morphodynamics Research Group and Department of Algebra and Geometry, Budapest University of Technology and Economics, Budapest, Hungary}
\email{{\tt zlangi@math.bme.hu}}

\keywords{Euclidean $d$-space, Minkowski $d$-space, packing, density, Voronoi decomposition in Minkowski plane, truncated Voronoi cell in Minkowski plane, Dowker-type theorems in Minkowski plane.}
\subjclass[2010]{52C15, 52C17, 52C05, 52C07, 52A40.}

\date{}

\begin{document}

\begin{abstract}
We prove that the highest density of non-overlapping translates of a given centrally symmetric convex domain relative to its outer parallel domain of given outer radius is attained by a lattice packing in the Euclidean plane. This generalizes some earlier (classical) results. Sharp upper bounds are proved for the analogue problem on congruent circular disks in the spherical (resp., hyperbolic) plane and on congruent balls in Euclidean $3$-space.
\end{abstract}

\maketitle

\section{Introduction}
 
 Packings and coverings, in particular, those with congruent balls, have a long history in mathematics, with many applications in sciences. Motivated by recent developments in natural sciences (see for example, \cite{edelsbrunner, senechal} for more information on such developments in molecular biology and crystallography), it seems that there is a need to broaden this investigation to families of convex bodies that are not packings, but in which some control on the overlap of the bodies prevents them to become coverings. These families are usually called \emph{soft packings}. In the recent papers \cite{Be-La-2015, Be-La-2024}, the authors investigated two notions of density related to soft packings of convex bodies in Euclidean spaces. In this paper we continue our investigation, and examine also soft packings in spaces of constant curvature. Before stating our results, we present a brief review on the background of our research.

Let us recall the classical notion of density for translative packings of convex bodies in Euclidean spaces. Let the Euclidean norm of the $d$-dimensional Euclidean space $\Eu^d$ be denoted by $\|\cdot\|$ and let $\BB^d:=\{\x\in \Eu^d | \|\x\|\leq 1\}$ denote the (closed) $d$-dimensional unit ball centered at the origin $\o$ in  $\Eu^d$, where $d>1$. Next,
let $\MM$ be an arbitrary $\o$-symmetric convex body (i.e., let $\MM$ be a compact, convex set with non-empty interior, and with $\MM=-\MM$) in $\Eu^d$, $d>1$ and let $\M$ be the $d$-dimensional Minkowski space (i.e., $d$-dimensional normed space) with $\MM$ as its unit ball, and norm $\| \x \|_{\MM}:=\min\{\mu>0 | \x\in\mu\MM\}$ for any $\x\in\Eu^d$.
Then let $\P:=\{\c_i+\MM\ |\ i\in I\subset\N \ {\rm with}\ \|\c_j-\c_k\|_{\M} \ge 2\ {\rm for}\ {\rm all}\ j\neq k\in I\}$ be an arbitrary {\it packing of unit balls} in $\M$ with $\mathbf{P}:=\bigcup_{i\in I}(\c_i+\MM)$. Finally, let
$$\delta_{\M}(\P):=\limsup_{R\to +\infty}\frac{\vol_d(\mathbf{P}\cap R\BB^d)}{\vol_d( R\BB^d) }$$
be the (upper) density of $\P$ in $\M$ and let
$$\delta_{\M}:=\sup_{\P}{\delta}_{\M}(\P) $$
be the {\it largest density} of packings of unit balls in $\M$, where $\vol_d(\cdot)$ denotes the $d$-dimensional volume (i.e., Lebesgue measure) of the corresponding set in $\Eu^d$. Computing $\delta_{\M}$ for $\M=\Eu^d$ is a long-standing open problem of discrete geometry with the exact values known only for $d=2$ (\cite{LFT50}), $3$ (\cite{L-05}), $8$ (\cite{V-17}), and $24$ (\cite{CKMRV}). (For more details see for example, Chapter 1 of \cite{Bezdek2013}). The concept of density has been extended to {\it soft packings} by the authors in cite{Be-La-2015} and \cite{Be-La-2024} as follows.  As above, let $\P:=\{\c_i+\MM\ |\ i\in I\subset\N \ {\rm with}\ \|\c_j-\c_k\|_{\M} \ge 2\ {\rm for}\ {\rm all}\ j\neq k\in I\}$ be an arbitrary packing of unit balls in $\M$. Moreover, for any $d\ge 2$  and $\lambda\ge 0$ let $\P_{\lambda}:=\{\c_i+(1+\lambda)\MM\ |\ i\in I\}$  with $\mathbf{P}_{\lambda}:=\bigcup_{i\in I}(\c_i+(1+\lambda)\MM)$ denoting the outer parallel domain of $\mathbf{P}:=\bigcup_{i\in I}(\c_i+\MM)$ for outer radius $\lambda$ in $\M$ (see Figure~\ref{fig:soft_packing}). Furthermore, let
$$\delta^{{\rm soft}}_{\M}(\P_{\lambda}):=\limsup_{R\to +\infty}\frac{\vol_d(\mathbf{P}_{\lambda}\cap R\BB^d)}{\vol_d(R\BB^d) }$$
be the (upper) density of the outer parallel domain $\mathbf{P}_{\lambda}$ assigned to the unit ball packing $\P$ in $\M$. Finally, let
$$\delta^{\rm soft}_{\M}(\lambda):=\sup_{\P}\delta_{\M}(\P_{\lambda})$$
be the largest density of the outer parallel domains of unit ball packings having outer radius $\lambda$ in $\M$. Putting it somewhat differently, one could say that the family $\P_{\lambda}=\{\c_i+(1+\lambda)\MM\ |\ i\in I\}$ of closed balls of radii $1+\lambda$ is a {\it packing of soft balls with soft parameter} $\lambda$  in $\M$, if $\P:=\{\c_i+\MM\ |\ i\in I\}$ is a unit ball packing  of $\M$ in the usual sense. In particular, $\delta^{\rm soft}_{\M}(\P_{\lambda})$ is called the (upper) {\it soft density} of the soft ball packing $\P_{\lambda}$  in $\M$ with
$\delta^{\rm soft}_{\M}(\lambda)$ standing for the {\it largest soft density} of packings of soft balls of radii $1+\lambda$ having soft parameter $\lambda$ in $\M$.

\begin{figure}[ht]
\begin{center}
\includegraphics[width=0.7\textwidth]{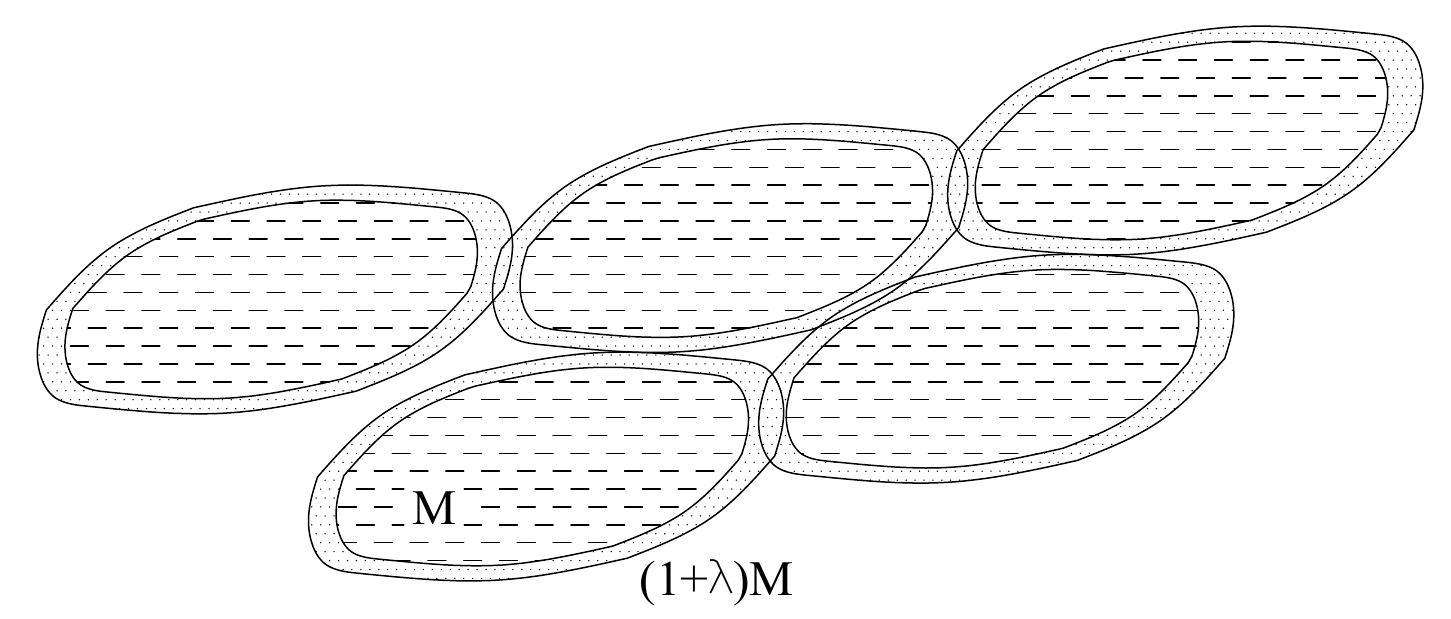}
\caption{A soft packing of translates of the $\o$-symmetric convex body $\MM$. The regions with horizontal stripes and dots correspond to $\mathbf{P}$ and $\mathbf{P}_{\lambda} \setminus \mathbf{P}$, respectively.}
\label{fig:soft_packing}
\end{center}
\end{figure}

The authors proved Rogers-type upper bounds for $\delta^{\rm soft}_{\Eu^d}(\lambda)$ in  \cite{Be-La-2015} (see Theorems 5 and 6 as well as Corollary 1), which are sharp for $d=2$ and $\lambda\geq 0$. On the other hand, the authors proved in \cite{Be-La-2024} (see Theorem 3) that for any $\lambda > 0$, among lattice packings of soft balls of radii $1+\lambda$ with soft parameter $\lambda$ in $\Eu^3$ (for an illustration of a soft lattice packing, see Figure~\ref{fig:soft_packing2}), the soft density has a local maximum at the corresponding FCC lattice.

\begin{figure}[ht]
\begin{center}
\includegraphics[width=0.7\textwidth]{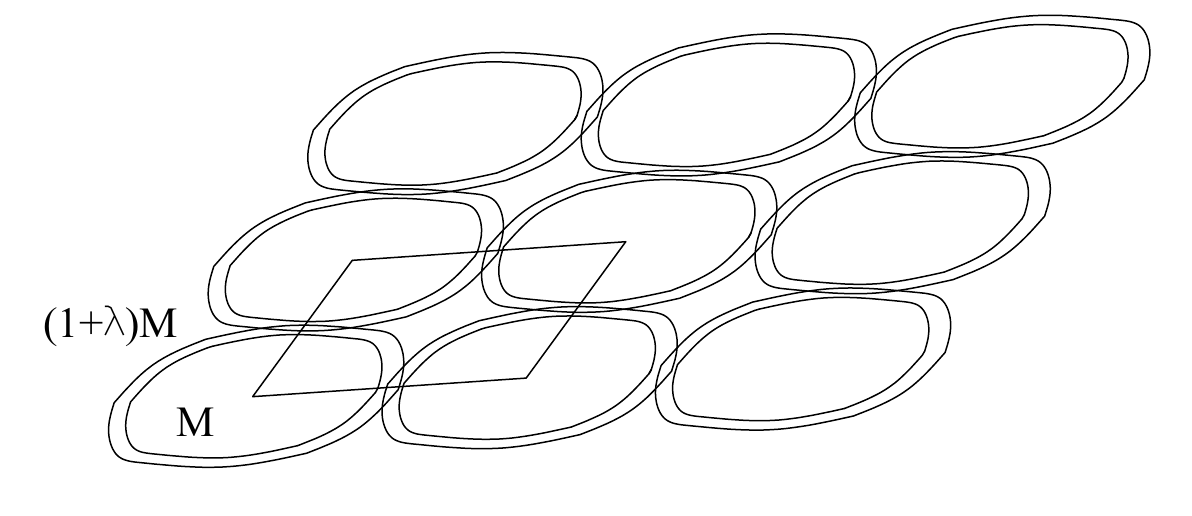}
\caption{A soft lattice packing of translates of the $\o$-symmetric convex body $\MM$, with a fundamental parallelogram of the lattice.}
\label{fig:soft_packing2}
\end{center}
\end{figure}

In this note we continue the line of research presented in  \cite{Be-La-2015} and   \cite{Be-La-2024}, and extend those results further by introducing the concept of {\it truncated density} $\delta_{\M}(\lambda)$ that bridges the concepts of density $\delta_{\M}$ and soft density $\delta^{\rm soft}_{\M}(\lambda)$ as follows.

\begin{definition}\label{defn:main}
Let $\MM$ be an arbitrary $\o$-symmetric convex body in $\Eu^d$, $d>1$ and let $\M$ be the $d$-dimensional Minkowski space (i.e., $d$-dimensional normed space) with $\MM$ as its unit ball, and norm $\| \x \|_{\MM}:=\min\{\mu>0 | \x\in\mu\MM\}$ for any $\x\in\Eu^d$. Then let $\P:=\{\c_i+\MM\ |\ i\in I\subset\N \ {\rm with}\ \|\c_j-\c_k\|_{\M} \ge 2\ {\rm for}\ {\rm all}\ j\neq k\in I\}$ be an arbitrary packing of unit balls in $\M$. Moreover, for any $d\ge 2$  and $\lambda\ge 0$ let $\mathbf{P}_{\lambda}:=\bigcup_{i\in I}(\c_i+(1+\lambda)\MM)$ denote the outer parallel domain of $\mathbf{P}:=\bigcup_{i\in I}(\c_i+\MM)$ having outer radius $\lambda$ in $\M$. Furthermore, let
$${\delta}_{\M}(\lambda , \P):=\limsup_{R\to +\infty}\frac{\vol_d(\mathbf{P}\cap R\BB^d)}{\vol_d(\mathbf{P}_{\lambda}\cap R\BB^d) }$$
be the (upper) {\rm $\lambda$-truncated density} of the unit ball packing $\P$ in $\mathbf{P}_{\lambda}$. Finally, let the largest $\lambda$-truncated density be defined by
$${\delta}_{\M}(\lambda):=\sup_{\P}{\delta}_{\M}(\lambda , \P) ,$$
where the supremum is taken for all packings $\P$ of unit balls in $\M$. Similarly, let $${\delta}_{\M}^{\mathrm{lattice}}(\lambda):=\sup_{\P_{\mathrm{lattice}}}{\delta}_{\M}(\lambda , \P_{\mathrm{lattice}}),$$
where $\P_{\mathrm{lattice}}$ stands for an arbitrary lattice packing of unit balls in $\M$. 
\end{definition}

\begin{remark} It follows from the above definitions in a straightforward way that
\begin{equation}\label{basic relation}
\delta_{\M}(\P)=\delta^{{\rm soft}}_{\M}(\P_{\lambda})\cdot {\delta}_{\M}(\lambda , \P)\ {\rm and}\ \delta_{\M}=\delta^{\rm soft}_{\M}(\lambda)\cdot{\delta}_{\M}(\lambda).
\end{equation}
For given $\o$-symmetric convex body $\MM$ in $\Eu^d$, the smallest $\lambda\geq 0$ for which $\delta^{\rm soft}_{\M}(\lambda)=1$, i.e., for which  $\delta_{\M}={\delta}_{\M}(\lambda)$, is closely related to the notion of simultaneous packing and covering constant of $\MM$, which has been extensively studied by several authors. We refer the interested reader to \cite{Zo2008} and the references mentioned there. Finally, we note that by taking a saturated packing of translates of $\MM$ in $\Eu^d$, it follows that $\delta^{\rm soft}_{\M}(\lambda)=1$ holds for all $\lambda\geq 1$. This together with (\ref{basic relation}) implies that
\begin{equation}\label{first}
\delta_{\M}={\delta}_{\M}(\lambda)
\end{equation}
holds for all $\lambda\geq 1$.
\end{remark}

\begin{remark} Clearly, if $\P:=\{\c_i+\MM\ |\ i=1,2,\dots,n \ {\rm with}\ \|\c_j-\c_k\|_{\M} \ge 2\ {\rm for}\ {\rm all}\ 1\leq j<k\leq n\}$ is a packing of $n\geq 1$ unit balls in $\M$, then $\delta_{\M}(\P)=\delta^{{\rm soft}}_{\M}(\P_{\lambda})=0$ and $ {\delta}_{\M}(\lambda , \P)>0$. In particular, for given $d>1, n>1$ and for any sufficiently small $\lambda>0$ the largest ${\delta}_{\Eu^d}(\lambda , \P)$ can be obtained for a packing $\P$ of $n$ unit balls in $\Eu^d$ that possesses the largest number of touching pairs. (For more details see Theorem 1 in  \cite{Be-La-2015}.)
\end{remark}

 \begin{remark} Clearly, $0<{\delta}_{\M}^{\mathrm{lattice}}(\lambda)\leq {\delta}_{\M}(\lambda)$ and ${\delta}_{\M}(\lambda)$ (resp., ${\delta}_{\M}^{\mathrm{lattice}}(\lambda)$) is a decreasing function of $\lambda\geq 0$ with ${\delta}_{\M}(0)={\delta}_{\M}^{\mathrm{lattice}}(0)=1$.
\end{remark} 

Independently, L. Fejes T\'oth \cite{LFT50} and C.A. Rogers \cite{Rogers} (see also \cite{Rogers2}) proved that 
\begin{equation}\label{second}
{\delta}_{\M}={\delta}^{\rm lattice}_{\M}
\end{equation} 
holds for any $\o$-symmetric convex body $\MM$ in $\Eu^2$, where ${\delta}^{\rm lattice}_{\M}$ denotes the largest (upper) density of lattice packings of $\MM$ in $\Eu^d$. This theorem, i.e., (\ref{second}) is a special case of the following more general result for sufficiently large $\lambda$ in (\ref{third}).
\begin{theorem}\label{thm:lattice}
For any $\o$-symmetric plane convex body $\MM$ and $\lambda\geq 0$, there is a lattice packing $\P_{\mathrm{lattice}}$ of translates of $\MM$ such that
$
\delta_{\M}(\lambda, \P_{\mathrm{lattice}}) = \delta_{\M}(\lambda).
$
In other words, we have
\begin{equation}\label{third}
\delta_{\M}(\lambda) = \delta_{\M}^{\rm lattice}(\lambda)
\end{equation}
for any $\o$-symmetric plane convex body $\MM$ and $\lambda\geq 0$.
\end{theorem}

Clearly, (\ref{first}), (\ref{second}), and (\ref{third}) imply 

\begin{corollary}\label{2D-lattice}
 ${\delta}_{\M}^{\mathrm{lattice}}(\lambda)={\delta}^{\rm lattice}_{\M}$ holds for all $\lambda\geq 1$ and for all $\o$-symmetric convex body $\MM$ in $\Eu^2$.
 \end{corollary}

\begin{remark}
Theorem~\ref{thm:lattice} for $\M=\Eu^2$, has already been proved in \cite{Be-La-2015} as follows. Let $0<\lambda<\frac{2}{\sqrt{3}}-1=0.1547...$ and $\mathbf{H}$ be a regular hexagon circumscribed about the unit disk $\BB^2$ in $\Eu^2$. Then
\begin{equation}\label{fourth}
{\delta}_{\Eu^2}(\lambda)=\frac{\pi}{\mathrm{area}\left(\mathbf{H}\cap(1+\lambda)\BB^2\right)}={\delta}_{\Eu^2}^{\mathrm{lattice}}(\lambda),
\end{equation}
where $\mathrm{area}(\cdot):=\vol_2(\cdot)$.
\end{remark}

Clearly, it is sufficient to prove Theorem~\ref{thm:lattice} under the assumption that $\MM$ has $C^{\infty}$-class boundary and it is strictly convex. We note that in this case the bisector of any nondegenerate segment $[\x, \y]$ is homeomorphic to a line (\cite{AKL}). We denote the bisector of $[\x,\y]$ by $B(\x,\y)$.

\begin{definition}
Let $\x_1,\x_2,\ldots, \x_k \in \M$ satisfy $||\x_i||_{\M} \geq 2$ for all $i$ with $1\leq i\leq k$. Then the set
\[
\PP = \{ \x : \|\x\|_{\M}  \leq \|\x-\x_i\|_{\M} \hbox{ for all } i \hbox{ with }1\leq i\leq k\} 
\]
is called a \emph{bisector $k$-gon} (shortly, B-$k$-gon) of $\MM$. For any $\lambda > 0$, we call the area of $\PP \cap (1+\lambda) \MM$ the \emph{$\lambda$-truncated area} of $\PP$.
\end{definition}

The proof of Theorem~\ref{thm:lattice} relies on the following Dowker-type theorem.

\begin{theorem}\label{thm:Dowker}
For any $\lambda > 0$ and $n \geq 3$, let $A_n(\MM,\lambda)$ denote the infimum of the $\lambda$-truncated areas of all B-$n$-gons of $\MM$. Then the sequence $\{ A_n(\MM,\lambda) \}$ is convex. Furthermore, if $k$ is a divisor of $n$ and $\MM$ has $k$-fold rotational symmetry, then there is a B-$n$-gon of $\MM$ with $k$-fold rotational symmetry whose $\lambda$-truncated area is $A_n(\MM,\lambda)$. In particular, if $n$ is even, there is an $\o$-symmetric B-$n$-gon whose $\lambda$-reduced area is $A_n(\MM,\lambda)$.
\end{theorem}

Next, we state a strengthening of (\ref{fourth}) as follows.
Let $\PPP$ be a plane of constant curvature (i.e., let $\PPP$ be either the Euclidean, or spherical, or hyperbolic plane), and let $\F$ be a packing of disks of radius $r$ in $\PPP$, Let $\BB$ be an element of $\F$ centered at $\q$, and let $\VV$ be the Voronoi cell of $\BB$. On the sphere, we also assume that the distance of any vertex of $\VV$ from $\q$ is less than $\frac{\pi}{2}$. Let $\BB_{\lambda}$ denote the disk of radius $(1+\lambda) r$ concentric with $\BB$. Then let the density of $\BB$ with respect to the truncated Voronoi cell $\BB_{\lambda} \cap \VV$ \big(resp., the soft density of the soft disk $\BB_{\lambda}$ with respect to $\VV$ having soft parameter $\lambda$ and hard core $\BB$\big) be defined by 
\[
 \delta(\BB,\VV,\lambda) := \frac{\area(\BB)}{\area(\BB_{\lambda} \cap \VV)}, \left(\mathrm{resp.,\ } \bar{\delta} (\BB,\VV,\lambda) := \frac{\area(\BB_{\lambda} \cap \VV)}{\area(\VV)} \right).
\]
For any $0 < r$ (on the sphere we also assume that $r < \frac{\pi}{3}$), let $R(r)$ denote the circumradius of a regular triangle $\TT_r$ of sidelengths $2r$. (Note that our assumption implies that on the sphere $R(r) < \frac{\pi}{2}$.)
Furthermore, if the vertices of $\TT_r$ are $\p_1,\p_2,\p_3$, and for $i=1,2,3$, $\BB_i$ and $\BB_i'$ denote the disks of radius $r$ and $(1+\lambda)r$ centered at $\p_i$, respectively, then we set
\[
\sigma_{reg}(r) := \frac{\area \left( \bigcup_{i=1}^3 \BB_i \cap \TT_r \right)}{\area\left( \bigcup_{i=1}^3 \BB_i' \cap \TT_r \right)},
\left(\mathrm{resp., \ } \bar{\sigma}_{reg}(r) = \frac{\area\left( \bigcup_{i=1}^3 \BB_i' \cap \TT_r \right)}{\area\left( \TT_r \right)}\right).
\]
We note that the distance of any sideline of $\VV$ from $\q$ is at least $r$, and the distance of any vertex of $V$ from the center of the disk of radius $r$ contained in $\VV$ is at least $R(r)$.

\begin{theorem}\label{thm:planeVoronoi}
With the above notation, we have
\[
\delta(\BB,\VV,\lambda) \leq \sigma_{reg}(r) {\mathrm{\ and\ }} \bar{\delta}(\BB,\VV,\lambda) \leq \bar{\sigma}_{reg}(r).
\]
\end{theorem}

\begin{remark}
We note that the first inequality of Theorem~\ref{thm:planeVoronoi} is a strengthening of (\ref{fourth}) in the Euclidean case and an extension of it to the spherical (resp., hyperbolic) plane. On the other hand, the second inequality of Theorem~\ref{thm:planeVoronoi} is an extension of the $2$-dimensional case of Theorem 5 in \cite{Be-La-2015} to planes of constant curvature.
\end{remark}

It would be very interesting to find analogues of the above stated $2$-dimensional results in higher dimensions. Here we state the following result in $\Eu^3$. Let $\DD$ be a regular dodecahedron circumscribed about the unit ball $\BB^3$, and note that the circumradius of $\DD$ is $R = \sqrt{3} \tan \frac{\pi}{5} = 1.2584\ldots$, and the radius of the midscribed ball of $\DD$ (touching the edges of $\DD$ at their midponts) is $r_{mid} = \frac{\sqrt{10-2\sqrt{5}}}{2} = 1.1755\ldots$.
For any $0 < \lambda \leq R-1$, let
\[
\tau_{\Eu^3}(\lambda) := \frac{\vol_3(\BB^3)}{\vol_3(\DD \cap (1+\lambda) \BB^3)}.
\]
One might wonder (see Problem 4 in \cite{Be-La-2015}) whether for any $0 < \lambda \leq R-1$, if $\VV$ is the Voronoi cell of $\BB^3$ in an arbitrary unit ball packing of $\Eu^3$, then
\[
\frac{\vol_{3}(\BB^3)}{\vol_3(\VV \cap (1+\lambda) \BB^3)} \leq \tau_{\Eu^3}(\lambda).
\]

\begin{theorem}\label{truncated dodecahedron}
For any $0 < \lambda \leq r_{mid}-1$, if $\VV$ is a Voronoi cell of $\BB^3$ in an arbitrary unit ball packing of $\Eu^3$, then
\[
\frac{\vol_3(\BB^3)}{\vol_3( \VV \cap (1+\lambda) \BB^3)} \leq \tau_{\Eu^3}(\lambda).
\]
\end{theorem}

\begin{remark}
On the one hand, we note that $\tau_{\Eu^3}(\lambda)=\frac{4\pi}{4\pi(1+\lambda)^3-36\pi(\lambda^2+\frac{2}{3}\lambda^3)}$ for all $0 < \lambda \leq r_{mid}-1=0.1755\dots$. On the other hand, Theorem 4 of \cite{Be-La-2015} proves the following: Let $0<\lambda\leq\frac{2}{\sqrt{3}}-1=0.1547\dots$ and set $\psi_0:=-\arctan\left(\sqrt{\frac{2}{3}}\tan(5\phi_0)\right)=0.0524\dots$, where $\phi_0:=\arctan\frac{1}{\sqrt{2}}=0.6154\dots$. If $\VV$ is a Voronoi cell of $\BB^3$ in an arbitrary unit ball packing of $\Eu^3$, then
\[
\frac{\vol_3(\BB^3)}{\vol_3(\VV \cap (1+\lambda) \BB^3)}
\] 

\[
\leq  \hat{\tau}_{\Eu^3}(\lambda):=\frac{\pi-6\psi_0}{\pi-6\psi_0+(3\pi-18\psi_0)\lambda-(6\pi+18\psi_0)\lambda^2-(5\pi+6\psi_0)\lambda^3}.
\]
(Unfortunately, the formula for $\hat{\tau}_{\Eu^3}(\lambda)$ in Theorem 4 of \cite{Be-La-2015} was computed incorrectly.) As $ \tau_{\Eu^3}(\lambda)<\hat{\tau}_{\Eu^3}(\lambda)$ holds for all $0<\lambda\leq\frac{2}{\sqrt{3}}-1=0.1547\dots$ therefore Theorem~\ref{truncated dodecahedron} of this paper improves Theorem 4 of \cite{Be-La-2015}.
\end{remark}

\begin{remark}
We note that the proof of Theorem 3 in \cite{Be-La-2024} implies in a straightforward way the following statement. For any $\lambda>0$, among lattice packings of unit balls in $\Eu^3$, the $\lambda$-truncated density has a local maximum at the corresponding FCC lattice.  
\end{remark}

We cannot resists to raise the following problem.

\begin{conjecture}\label{truncated Kepler}
For all $\lambda>0$, among packings of unit balls in $\Eu^3$, the $\lambda$-truncated density has a maximum at the corresponding FCC lattice.
\end{conjecture}

\begin{remark}
Hales's celebrated proof \cite{L-05} of the Kepler conjecture and (\ref{first}) imply Conjecture \ref{truncated Kepler} for any $\lambda\geq 1$. On the other hand, Conjecture \ref{truncated Kepler} seems to be open for $0<\lambda< 1$. For answering the latter question, the following inequality of Hales \cite{L-12} might be helpful: Let $\alpha_0 := 1.26$, and $\x_i$, $i=1,2,\ldots, m$ be a set of points in $\Eu^3$ satisfying $2 \leq \| \x_i \| \leq 2 \alpha_0$ and $2 \leq \| \x_i - \x_j \|$ for any $i \neq j$. Then $24+2(m-12)\alpha_0\leq\sum_{i=1}^{m}\|\x_i\|$. (For a related concept we refer the interested reader to
\cite{BeLi}.)
\end{remark}

In the rest of this note we prove the theorems stated.

\section{Proofs of Theorems~\ref{thm:lattice} and \ref{thm:Dowker}}

To prove the theorems, we use the description of the properties of Voronoi cells of point systems in normed planes in e.g. \cite{AKL}. In particular, we note that as $\MM$ is strictly convex and it has a $C^{\infty}$-class boundary, for any distinct $\x, \y \in \M$, $B(\x,\y)$ is a $C^{\infty}$-class, simple curve. For the proofs of our theorems, we need Lemma~\ref{lem:Voronoi}.

\begin{lemma}\label{lem:Voronoi}
For any mutually distinct points $\x,\y,\z$, $B(\x,\y)$ and $B(\y,\z)$ intersect if and only if $\x,\y,\z$ are not collinear, and in this case they have a unique intersection point.
\end{lemma}

\begin{proof}
The proof is based on the following reformulation of the statement: if $\x,\y,\z$ are collinear, then there is no homothetic copy of $\MM$ containing all of them in its boundary, and if they are not collinear, there is a unique such copy. To prove this statement, it is enough to observe that for any distinct points $\x,\y$, there is a $1$-parameter family of homothetic copies of $\MM$ containing them in their boundaries.
\end{proof}

First, we show how Theorem~\ref{thm:Dowker} yields Theorem~\ref{thm:lattice}.

Let $\varepsilon > 0$ be arbitrary. Consider a packing $\P = \{ \c + \MM : \c \in X \}$ of translates of $\MM$ such that the density of $X+\MM$ with respect to its outer parallel domain $X+ (1+\lambda) \MM$ of outer radius $\lambda$ is at least $\delta_{\M}(\lambda) - \varepsilon$. Note that adding translates of $\MM$ to $\P$ may decrease its density, and thus, we cannot assume that it is a saturated packing. Nevertheless, we show that there is some $r > 0$ and a modified packing $\P'$ with density at least $\delta_{ \M}(\lambda) - \varepsilon$ such that for any $\x \in { \M}$, $\| \x -\c \|_{ \M} \leq 2r$ for some $\c \in X$.

Indeed, for any $r > 2$, let $X_r = X \cap (r-1) \MM$. Clearly, we may choose some sufficiently large value $r'$ such that the density of $X_{r'} + \MM$ with respect to its outer parallel domain of outer radius $\lambda$ is at least $\delta_{\M}(\lambda) - \varepsilon$. Choose some countable set $Y \subset {\M}$ such that $\| \y - \y' \|_{\M} \geq 2r'$ for any distinct $\y, \y' \in Y$, and $ Y + 2r'\MM = \M$. We define $X' = \bigcup_{\y \in Y} (\y+ X_{r'})$. Then $\P = \{ \c + \MM : \c \in X' \}$ is a packing of translates of $\MM$. Furthermore, the density of this packing is at least $\delta_{\M}(\lambda) - \varepsilon$, and for any $\z \in {\M}$, we have $\| \z-\c \|_{ \M} \leq 2r'$ for some $\c \in X'$.

In the remaining part, we show that if $\P = \{ \c + \MM : \c \in X \}$ is a packing of translates of $\MM$ such that for some $r > 0$, $X + r \MM = {\M}$, then
\begin{equation}\label{eq:latticedens}
\frac{\area(\MM)}{A_6(\MM, \lambda)} \geq \delta_{\M}(\lambda, \P).
\end{equation}
Here we observe that by Theorem~\ref{thm:Dowker}, there is an $\o$-symmetric B-$6$-gon circumscribed about $\MM$ with area equal to $A_6(\MM, \lambda)$. The sides of this B-$6$-gon belong to bisectors of the form $B(\o, \pm \x)$, $B(\o, \pm 2\y)$ and $B(\o, \pm( \x +\y) )$ for some linearly independent vectors $\x, \y \in \M$, implying that the packing $\P_{lattice} = \{ k \x + m \y + \MM : k,m \in \mathbb{Z} \}$ is a lattice packing of $\MM$ with density $\frac{\area(\MM)}{A_6(\MM, \lambda)}$, and thus, (\ref{eq:latticedens}) implies Theorem~\ref{thm:lattice}.

Let us decompose $\M$ into the Voronoi cells of the elements of $\P$, where for any $\c_i \in X$, $\CC_i$ denotes the cell of $\c_i$ in this decomposition. Furthermore, by our assumptions and Lemma~\ref{lem:Voronoi}, $\CC_i$ is a starlike topological disk, and its boundary is a closed chain of arcs such that each arc is a connected part of a bisector $B(\c_i,\c_j)$ for some $\c_j \in X$. These arcs are called the \emph{sides} of $\CC_i$, and the unique intersection point of two consecutive sides is a \emph{vertex} of $\CC_i$. This also shows that the decomposition is edge-to-edge. Let $n_i$ denote the number of sides of $\CC_i$, and $A_i$ denote the area of the region $\CC_i \cap \left( \c_i +(1+\lambda) \MM \right)$.

For any $R > 0$, let $\P_R$ denote the family of Voronoi cells contained in $R \BB^2$, and let $N_R$ denote the cardinality of $\P_R$.
Since for any cell $\CC_i$, $\c_i + \MM \subseteq \CC_i \subseteq \c_i + r \MM$, it is easily deduced that the limit $n:=\lim_{R \to \infty} \frac{\sum_{\CC_i \in \P_R} n_i}{N_R }$ exists. Furthermore, as the edge graph of the Voronoi decomposition, formed by the vertices and sides of the cells, is planar, by Euler's theorem we obtain that $n \leq 6$. 
Then, by Theorem~\ref{thm:Dowker},
\begin{equation}\label{eq:lattice}
\frac{\sum_{\CC_i \in \P_R} A_i}{N_R} \geq \frac{\sum_{\CC_i \in \P_R} A_{n_i}(\MM,\lambda)}{N_R} \geq A_6(\MM, \lambda),
\end{equation}
which readily yields the inequality in (\ref{eq:latticedens}).

\begin{figure}[ht]
\begin{center}
\includegraphics[width=0.7\textwidth]{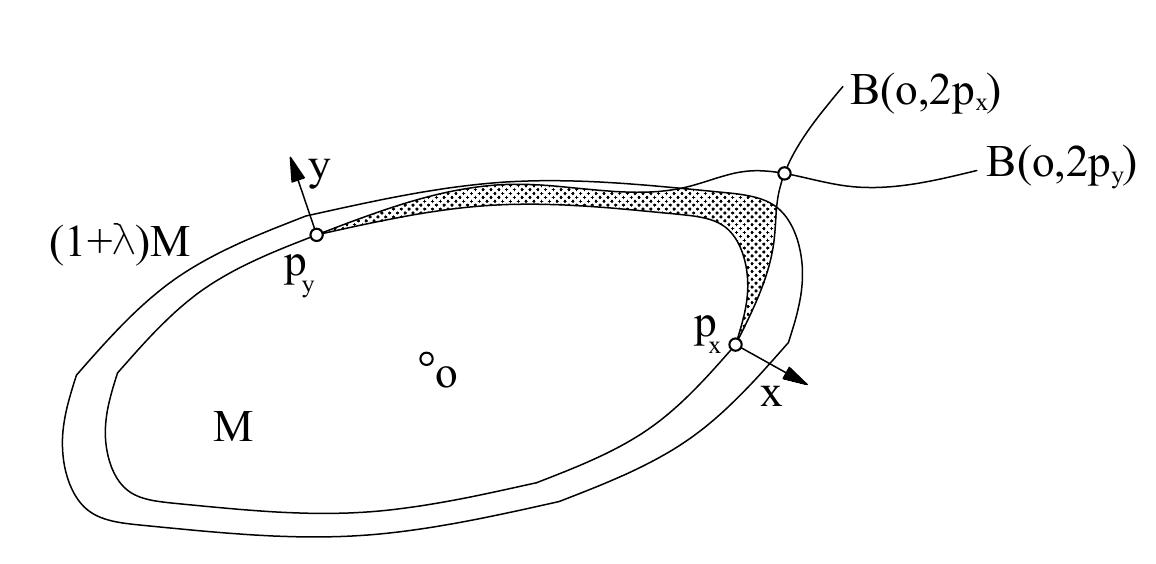}
\caption{An illustration for the proof of Theorem~\ref{thm:Dowker}. The area of the dotted region is $A(\widehat{\x\y})$.}
\label{fig:B-k-gon}
\end{center}
\end{figure}

In the remaining part of the section, we prove Theorem~\ref{thm:Dowker}.
First, observe that, by compactness, for any $n$ there is a B-$n$-gon $\PP$ of minimal $\lambda$-reduced area containing $\MM$, and every side of $\PP$ touches $\MM$.
Let $\mathcal{S}$ denote the family of closed oriented arcs (in counterclockwise direction) in the Euclidean unit circle $\mathbb{S}^1$, where $\widehat{\x\y}$ denotes the closed counterclockwise oriented arc in $\mathbb{S}^1$ from $\x$ to $\y$.
For any $\x \in \mathbb{S}^1$, let $\p_x$ denote the unique point of $\bd (\MM)$ with outer unit normal vector $\x$. For any $\x, \y \in \mathbb{S}^1$, let 
$R(\widehat{\x\y})$ denote the (closed) region bounded by the counterclockwise oriented arc in $\bd (\MM)$ from $\p_{x}$ to $\p_y$ and the bisectors $B(\o,2\p_x)$, $B(\o,2\p_y)$ touching $\MM$ at $\p_{x}$ and $\p_y$, respectively, such that $\o \notin R(\widehat{\x\y})$ (see Figure~\ref{fig:B-k-gon}). Finally, define $A(\widehat{\x\y})$ as the area of $(1+\lambda) \MM \cap R(\widehat{\x\y})$. We recall Lemma 5 from \cite{BL23}.

\begin{lemma}\label{lem:functional}
Let $f : \mathcal{S} \to \Re$ be a bounded function with $f(\widehat{pp})=0$ for all $p \in \S^1$. For any integer $n \geq 3$, let
\[
m_n = \inf  \left\{  \sum_{S \in X} f( S ) : X \subset \mathcal{S} \hbox{ is a tiling of } \S^1 \hbox{ with } |X| = n \right\}.
\]
If for any $\widehat{x_2x_3} \subset \widehat{x_1x_4}$, we have
\[
f(\widehat{x_1x_3})+f(\widehat{x_2x_4}) \leq f(\widehat{x_1x_4})+f(\widehat{x_2x_3}),
\]
then the sequence $\{ m_n \}$ is convex.
Furthermore, if in addition, there is some positive integer $k$ that divides $n$, and $f$ has $k$-fold rotational symmetry, and there is an $n$-element tiling  $X$ of $\S^1$ such that $m_n = \sum_{S \in X} f(S)$, then there is an $n$-element tiling $X'$ of $\S^1$ with $k$-fold rotational symmetry such that $m_n = \sum_{S \in X' } f(S)$.
\end{lemma}

Let $I[X] : \Eu^2 \to \{ 0,1\}$ denote the indicator function of the set $X \cap (1+\lambda) \MM$ for any $X \subseteq \Eu^2$. Consider arcs $\widehat{x_2x_3} \subset \widehat{x_1x_4} \subset \mathbb{S}^1$. Then it is easy to check that
\[
I[R(\widehat{\x_1\x_4})](\q)+ I[R(\widehat{\x_2\x_3})](\q) - I[R(\widehat{\x_1\x_3})](\q) - I[R(\widehat{\x_2\x_4})](\q) \geq 0
\]
is satisfied for any $\q \in \Eu^2$. This shows that the conditions in Lemma~\ref{lem:functional} are met for the function $f(\cdot) = A(\cdot)$, implying Theorem~\ref{thm:Dowker}.

\section{Proof of Theorem~\ref{thm:planeVoronoi}}

In the proof, for any points $\p, \q \in \PPP$, where on the sphere we assume that they are not antipodal, we denote by $[\p,\q]$ the shortest geodesic arc connecting $\p$ and $\q$.  For any triangle $\TT$ with $\q$ as a vertex, we set
\[
\rho(\BB,\TT) := \frac{\area(\TT \cap \BB)}{\area(\TT \cap \BB_{\lambda})}, \hbox{ and } \hat{\rho}(\BB,\TT) := \frac{\area(\TT \cap \BB_{\lambda})}{\area(\TT)}.
\]

Before proving the theorem, we first prove two lemmas.

\begin{lemma}\label{lem:constcurv1}
Let $L$ be a line through $\q$, and let $R_p$ be a half line starting at a point $\p$ of $L \setminus \inter (\BB)$ such that $R_p$ is perpendicular to $L$. For any $s >0 $ (on the sphere for any $0 < s < \frac{\pi}{2})$, let $\p(s)$ denote the point of $R_p$ whose distance from $\p$ is $s$. Furthermore, for any $0 < s_1 < s_2$ (on the sphere for any $0 < s_1 < s_2 < \frac{\pi}{2})$, let $\TT(s_1,s_2)$ denote the triangle with vertices $\q, \p(s_1)$ and $\p(s_2)$. Then $\rho(\BB,\TT(s_1,s_2))$ and $\hat{\rho}(\BB,\TT(s_1,s_2))$ are non-increasing functions of $s_1$ and $s_2$.
\end{lemma}

\begin{proof}
Let $0  < s_1 < s_2 < s_2'$ be given (on the sphere assume also that $s_2' < \frac{\pi}{2}$). Without loss of generality, assume that the ratio of the angles of $\TT(s_1,s_2)$ and $\TT(s_1,s_2')$ at $\q$ is rational. Then we can dissect $\TT(s_1,s_2')$ into triangles $\TT_1, \TT_2, \ldots, \TT_k$, all having $\q$ as a vertex such that they have equal angles at $\q$, denoted by $\gamma$, any two consecutive members share a side, and $\p(s_2)$ is a vertex of some triangles $\TT_i$ and $\TT_{i+1}$. Let $\varphi : \PPP \to \PPP$ be the rotation around $\q$ with angle $\gamma$ that moves $\p(s_1)$ closer to $\p(s_2)$. Then, for any value of $i$, $\varphi(T_i) \subseteq T_{i+1}$, $\varphi(\TT_i \cap \BB) = \TT_{i+1} \cap \BB$ and $\varphi(\TT_i \cap \BB_{\lambda}) = \TT_{i+1} \cap \BB_{\lambda}$.
This proves that $\rho(\BB,\TT(s_1,s_2)) \geq \rho(\BB,\TT(s_1,s_2'))$, implying that $\rho(\BB,\TT(s_1,s_2))$ is a non-increasing function of $s_2$. To prove that it is a non-increasing function of $s_1$, we may apply a similar argument.

To verify the monotonicity property of $\hat{\rho}(\BB,\TT(s_1,s_2))$, note that there is a unique value $\bar{s}$ such that $\p(s) \in \BB_{\lambda}$ if and only if $s \leq \bar{s}$. Furthermore, if $0 < s_1 < s_2 \leq \bar{s}$, then $\hat{\rho}  (\BB,\TT(s_1,s_2)) =1$. Thus, to prove that $\hat{\rho}(\BB,\TT(s_1,s_2))$ is a non-increasing function of $s_2$, it is sufficient to consider values $\bar{s} \leq s_1 < s_2 < s_2'$. But in this case, repeating the argument in the previous paragraph, we obtain
$\hat{\rho}(\BB,\TT(s_1,s_2)) \geq \hat{\rho}(\BB,\TT(s_1,s_2'))$, showing that the function is non-increasing in the second variable. To prove the monotonicity in the first variable, we can repeat the same argument.
\end{proof}

\begin{figure}[ht]
\begin{center}
\includegraphics[width=0.52\textwidth]{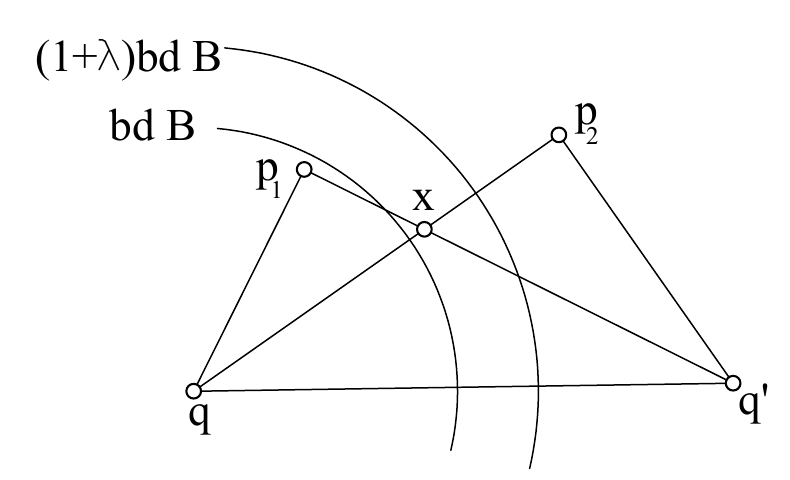}
\caption{An illustration for the proof of Lemma~\ref{lem:constcurv2}.}
\label{fig:constcurv}
\end{center}
\end{figure}

\begin{lemma}\label{lem:constcurv2}
Let $\q'$ be a point at distance $R(r)$ from $\q$. For $i=1,2$, let $\TT_i = \conv \{ \q, \q', \p_i \}$ be a triangle with right angle at $\p_i$, let $r_i$ denote the distance of $\p_i$ and $\q$. Assume that $r \leq r_1 < r_2 < R(r)$. Then
\[
\rho(\BB,\TT_1)  \geq \rho(\BB,\TT_2), \hbox{ and } \hat{\rho}(\BB,\TT_1) \geq \hat{\rho}(\BB,\TT_2).
\] 
\end{lemma}

\begin{proof}
It is easy to show that $r_1 < r_2$ implies that the angle of $\TT_1$ at $\q$ is larger than that of $\TT_2$, and thus we may assume that $[\q,\p_2]$ and $[\q',\p_1]$ cross. Let the intersection point of these segments be $\x$ (see Figure~\ref{fig:constcurv}). Let $\TT_a=\conv \{ \q,\p_1,\x \}$, $\TT_b = \conv \{ \q, \x, \q' \}$ and $\TT_c = \conv \{ \x, \q', \p_2 \}$. Then, by Lemma~\ref{lem:constcurv1}, we have
\[
\rho(\BB,\TT_a) \geq \rho(\BB,\TT_b) \geq \rho(\BB,\TT_c),
\]
implying that $\rho(\BB,\TT_1) \geq \rho(\BB,\TT_2)$. By the same lemma we similarly obtain that $\hat{\rho}(\BB,\TT_a) \geq \hat{\rho}(\BB,\TT_b)$. Furthermore, an argument similar to the proof of Lemma~\ref{lem:constcurv1} yields $\hat{\rho}(\BB,\TT_b) \geq \frac{\area(\BB_{\lambda} \cap \TT_c)}{\area(\TT_c)}$, which readily implies the required inequality for $\hat{\rho}$.
\end{proof}

Now we prove Theorem~\ref{thm:planeVoronoi}.
Let us dissect the Voronoi cell $\VV$ of $\q$ into triangles by connecting $\q$ to all vertices of $\VV$. Furthermore, if the orthogonal projection of $\q$ onto any side of $\VV$ lies inside the side, we dissect the corresponding triangle obtained in the previous step into two right triangles by connecting $\q$ to this ortogonal projection. Note that any triangle obtained in this way is either obtuse, or a right triangle.

Let $\TT$ be a triangle obtained in the previous paragraph. Let $r_T$ denote the distance of $\q$ from the sideline of $\TT$ opposite to it, and let $d_T$ denote the distance of $\q$ from the farthest vertex of $\TT$. Note that $r_T \geq r$ and $d_T \geq R(r)$.
By Lemma~\ref{lem:constcurv1}, we have that fixing the value of $r_T$, $\rho(\BB,\TT)$ and $\hat{\rho}(\BB,\TT)$ are maximized if $\TT$ is a right triangle, and $d_T= R(r)$. Under these conditions, by Lemma~\ref{lem:constcurv2}, $\rho(\BB,\TT)$ is maximal if $r_T=r$. But in this case $\rho(\BB,\TT) = \sigma_{reg}(r)$ and $\hat{\rho}(\BB,\TT) = \bar{\sigma}_{reg}(r)$, implying Theorem~\ref{thm:planeVoronoi}.

\section{Proof of Theorem~\ref{truncated dodecahedron}}

Without loss of generality, we may assume that the distance of every face plane of $\VV$ from $\o$ is at most $1+\lambda$.
Let the face-planes of $\VV$ be $H_i$ with $i=1,2,\ldots, m$. For any $H_i$, let $\CC_i$ denote the spherical segment of $(1+\lambda) \BB^3$ truncated by $H_i$ from this ball.
To minimize the density, we need to maximize $\vol_3 \left( \bigcup_{i=1}^m \CC_i \right)$. Let the height of $\CC_i$, defined as $(1+\lambda)$ minus the distance of $H_i$ from $\o$, be denoted by $h_i$.
Note that 
\[
\vol_3 \left( \bigcup_{i=1}^m \CC_i \right) \leq \sum_{i=1}^m \vol_3(\CC_i).
\]
A well-known formula shows that $\vol_3(\CC_i) = \pi \left( (1+\lambda) h_i^2 - \frac{h_i^3}{3} \right)$. We set
\[
F(h_1,h_2,\ldots, h_m) := \sum_{i=1}^m \pi \left( (1+\lambda) h_i^2 - \frac{h_i^3}{3} \right).
\]
Consider the function $f(h) := \pi \left( (1+\lambda) h^2 - \frac{h^3}{3} \right)$. Then $f'(h)=\pi h(2(1+\lambda)-h)>0$ and $f''(h) = \pi (2(1+\lambda)-2h)>0$ for all $0<h<1+\lambda$. It follows that $f$ is a
strictly increasing and convex function of $h$ on the interval $[0,\lambda]$.
Hence, if $\sum_{i=1}^m h_i \leq 12 \lambda$, then
\begin{equation}\label{eq:first}
F(h_1,h_2,\ldots, h_m)=\sum_{i=1}^m f(h_i)  \leq 12 \pi \left( (1+\lambda) \lambda^2 - \frac{\lambda^3}{2} \right).
\end{equation}
Note that the right-hand side of (\ref{eq:first}) is $\vol_3 ((1+\lambda) \BB^3 \setminus \DD)$, implying Theorem~\ref{truncated dodecahedron} in this case.

In the remaining part we show that $\sum_{i=1}^m h_i \leq 12 \lambda$. Suppose the contrary, and note that this implies $m > 12$. We recall the following lemma of Hales from \cite{L-12}.

\begin{lemma}\label{lem:Hales}
Let $\alpha_0 := 1.26$, and $\x_i$, $i=1,2,\ldots, m$ be a set of points in $\Eu^3$ satisfying $2 \leq \| \x_i \| \leq 2 \alpha_0$ and $2 \leq \| \x_i - \x_j \|$ for any $i \neq j$.
Then
\[
\sum_{i=1}^m L \left( \frac{1}{2} \| \x_i \| \right) \leq 12,
\]
where $L(t) := \frac{\alpha_0-t}{\alpha_0-1}$ for $0\leq t\leq\alpha_0$.
\end{lemma}

Applying Lemma~\ref{lem:Hales} for the centers of the balls generating the faces of $\VV$, we have
\[
12 \geq \sum_{i=1}^m L(1+\lambda-h_i) = \frac{\alpha_0-1-\lambda}{\alpha_0-1} m +\frac{1}{\alpha_0-1} \sum_{i=1}^m h_i.
\]
Combining it with the indirect assumption, we obtain
\begin{equation}\label{eq:estimate1}
12 \lambda < \sum_{i=1}^m h_i \leq 12 (\alpha_0-1)-m(\alpha_0-1-\lambda).
\end{equation}
From this, a rearrangement of the terms in the leftmost and rightmost expressions yields
\[
(m-12)(\alpha_0-1-\lambda) \leq 0,
\]
which contradicts the fact that both factors on the left-hand side are positive.

\vskip0.5cm
{\bf Acknowledgement.}
{\small
K.B. was partially supported by a Natural Sciences and Engineering Research Council of Canada Discovery Grant and Z.L. was partially supported by the National Research, Development and Innovation Office, NKFI, K-147544 and BME Water Sciences and Disaster Prevention TKP2020 Institution Excellence Subprogram, grant no. TKP2020 BME-IKA-VIZ.
}

\end{document}